\documentclass{amsart}
	\usepackage{aliascnt}
	\usepackage[colorlinks=true, linkcolor=black, citecolor=black, urlcolor=black]{hyperref}
	\usepackage{amssymb}	
	\usepackage{amsmath}
	\usepackage{amsthm}
	\usepackage{layout}

	\newaliascnt{lemma}{thm}
	\newtheorem{lemma}[lemma]{Lemma}  
	\aliascntresetthe{lemma}

	\newaliascnt{prop}{thm}
	\newtheorem{prop}[prop]{Proposition} 
	\aliascntresetthe{prop}

	\newaliascnt{cor}{thm}
	
	\aliascntresetthe{cor}

	\theoremstyle{remark}

	\newaliascnt{rem}{thm}
	\newtheorem{rem}[rem]{Remark}
	\aliascntresetthe{rem}

	\theoremstyle{definition}

	\newaliascnt{exm}{thm}
	
	\aliascntresetthe{exm}

	\newaliascnt{notn}{thm}
	
	\aliascntresetthe{notn}

	\newaliascnt{defn}{thm}
	\newtheorem{defn}[defn]{Definition}
	\aliascntresetthe{defn}

	\newcommand{\K}{\mathbb{K}}
	\newcommand{\rk}{\operatorname{rk}}
	\newcommand{\rmax}{\operatorname{r_{max}}}
	
	\newcommand{\Sy}{\operatorname{Sym}}
	\newcommand{\Ker}{\operatorname{Ker}}
	\newcommand{\Hom}{\operatorname{Hom}}
	\newcommand{\Ps}{\mathbb{P}}
	\newcommand{\Span}[1]{\left\langle\,#1\,\right\rangle}
	\newcommand{\blen}{\operatorname{b\ell}}
	
\begin{document}

\title[$\ldots$ leading term for maximum rank of ternary forms $\ldots$]{The asymptotic leading term for maximum\\rank of ternary forms of a given degree}
\author{Alessandro De Paris}
\address{Dipartimento di Matematica e Applicazioni ``Renato Caccioppoli'',\newline\indent Universit\`a di Napoli Fe\-de\-ri\-co II (Italy)}
\email{deparis@unina.it}

\begin{abstract}
Let $\rmax(n,d)$ be the maximum Waring rank for the set of \emph{all} homogeneous polynomials of degree $d>0$ in $n$ indeterminates with coefficients in an algebraically closed field of characteristic zero. To our knowledge, when $n,d\ge 3$, the value of $\rmax(n,d)$ is known only for $(n,d)=(3,3),(3,4),(3,5),(4,3)$. We prove that $\rmax(3,d)=d^2/4+O(d)$ as a consequence of the upper bound $\rmax(3,d)\le\left\lfloor\left(d^2+6d+1\right)/4\right\rfloor$.
\end{abstract}

\maketitle

\textbf{Keywords:} Waring problem, rank, symmetric tensor.

\medskip
\textbf{MSC2010:} 15A21, 15A69, 15A72, 14A25, 14N05, 14N15.

\section{Introduction}

A natural kind of Waring problem asks for the least of the numbers $r$ such that \emph{every} homogeneous polynomial of degree $d>0$ in $n$ indeterminates can be written as a sum of $r$ $d$th powers of linear forms. For instance, when $(n,d)=(3,4)$ (and the coefficients are taken in an algebraically closed field of characteristic zero), the answer is $7$. This was found for the first time in \cite{K}. In view of the interplay with the rank of tensors, relevant applicative interests of questions like this have recently been recognized (see \cite{L}). For further information we refer the reader  to \cite[Introduction]{T}.

Every power sum decomposition gives rise to a set of points in the projectivized space of linear forms, and in \cite{D1} it is proved that for ternary quartics one can always obtain a power sum decomposition by considering seven points arranged along three lines. In \cite{D2}, considering sets of points arranged along four lines, one finds that every ternary quintic is a sum of $10$ fifth powers of linear forms. Ternary quintics without power sum decompositions with less than $10$ summands were exhibited soon after in~\cite{BuT}. Hence, the answer in the case $(n,d)=(3,5)$ is $10$.

In the present paper we test ``at infinity'' the technique of arranging decompositions of ternary forms along lines. More precisely, let $\rmax(n,d)$ denote the desired answer to the mentioned Waring problem. Elementary considerations show that for each fixed $n$, $\rmax(n,d)=O\left(d^{n-1}\right)$, and if $\rmax(n,d)=c_nd^{n-1}+O\left(d^{n-2}\right)$ for some constant $c_n$ (as it is reasonable to expect), then it must be $1/n!\le c_n\le 1/(n-1)!$.  The best general upper bound on $\rmax(n,d)$ to our knowledge is given by \cite[Corollary~9]{BlT}. This implies that the constant $c_n$ (if it exists) is at most $2/n!$. Using \cite[Proposition~4.1]{CCG} (see also \cite[Theorem~7]{BBT}, \cite[Theorem~1]{BuT}), we deduce $\rmax(3,d)\ge\left\lfloor\left(d+1\right)^2/4\right\rfloor$. Hence, it must be $1/4\le c_3\le 1/3$. In the present work, for all ternary forms of degree $d$ we obtain power sum decompositions by considering $\left\lfloor\left(d^2+6d+1\right)/4\right\rfloor$ points arranged along $d$ lines. Hence, we have $\rmax(3,d)=d^2/4+O(d)$, that is, $c_3=1/4$.

The upper bound we are proving lowers the general upper bound \cite[Corollary~9]{BlT} in the special case $n=3$ and for $d\ge 6$. Nevertheless, it is not the best we can achieve because our purpose here was to determine the asymptotic leading term as simply as we could. To explain how the method works and why the resulting bound can ulteriorly be lowered, let us consider what happens for a ternary quartic~$f$. For introductory purposes, we now use a geometric language; the technical heart of the paper will be elementary linear algebraic instead. We view our quartic as a point $\Span{f}$ in the $14$-dimensional projective space of all quartic forms, where fourth powers make a degree $16$ Veronese surface. That surface is isomorphic to a plane via quadruple embedding, and exploiting apolarity we get four lines, which embed as rational normal quartics. The four curves are chosen so that their span contains $\Span{f}$, but no three of them do the same. Then, by means of successive projections and liftings we get a sequence of essentially binary forms that easily handle power sum decompositions. More precisely, we successively consider decompositions of binary forms of degrees $1,2,3,4$, with respective lengths $2,2,3,3$. Thus $\rk f\le 10$. This bound is rather relaxed since $\rmax(3,4)=7$. Note, however, that generic ranks of binary forms of degrees $1,2,3,4$ are $1,2,2,3$. Moreover, if one uses \cite[Proposition~2.7]{BD} instead of \autoref{Lines} here, one gets three (or fewer) lines instead of four. With three lines, the binary forms involved are of degrees $2,3,4$, and the corresponding generic ranks are $2,2,3$. This way, with a few additional technical cautions, we can reach the value of $\rmax(3,4)$. Similarly, we can reach $\rmax(3,5)=10$ in a simpler way than in \cite{D2}. For ternary sextics and septics, it is reasonabe to expect that the bounds $\rmax(3,6)\le 14$, $\rmax(3,7)\le 18$ can be proved with a more or less straightforward extension of the method.  However, in the present work we prefer not to set up in detail these results about low-degree forms because there are also reasons to believe that to reach $\rmax(3,d)$, further considerations could be in order (maybe an enhanced choice of the lines, if not a completely different strategy). We now outline what these reasons are.

When the present paper was in preparation, a log cabin patchwork like the following was shown to us (\footnote{As strange as it seems, during the lunch break on October 7, 2015, the TV was on and at a certain point the patchwork was shown as a tutorial about sewing in the program ``Detto Fatto'', broadcast by the national Italian channel RAI 2.}):
\begin{figure}[h]
\setlength{\unitlength}{.5cm}
\begin{picture}(5,5)
\put(0,0){\line(1,0){5}}
\put(5,0){\line(0,1){5}}
\put(5,5){\line(-1,0){4}}
\put(0,1){\line(0,-1){1}}
\put(0,1){\line(1,0){4}}
\put(4,0){\line(0,1){4}}
\put(5,4){\line(-1,0){3}}
\put(1,5){\line(0,-1){4}}
\put(1,2){\line(1,0){2}}
\put(3,1){\line(0,1){2}}
\put(4,3){\line(-1,0){2}}
\put(2,5){\line(0,-1){3}}
\end{picture}
\caption{}
\end{figure}

The area of the patches, starting from the center, makes a sequence \[1,2,2,2,3,3,4,4,\ldots\;.\] The partial sums are \[1,3,5,7,10,13,17,21,\ldots\;.\] The first five partial sums agree with the values of $\rmax(3,d)$, $d=1,\ldots,5$ that are known at the time of writing. The picture also clearly shows that the area is asymptotically $d^2/4$. This suggests that $\rmax(3,d)$ could be $\left\lfloor\left(d^2+2d+5\right)/4\right\rfloor$ for $d\ge 2$. This would mean that \cite[Theorem~1]{BuT} is the best that one can achieve for $n=3$ and odd $d\ge 3$, and that for even degrees one should be able to raise by one the rank reached by monomials (like \cite[Theorem~1]{BuT} for odd degrees).

The picture suggests how to build sets of points that may give rise to forms with the desired rank. On the other hand, at the moment we do not know how our technique for upper bounds could be improved. To decide for the values $12\le\rmax(3,6)\le 14$ and $17\le\rmax(3,7)\le 18$, would indicate whether or not the ``patchwork conjecture'' is more promising than a likewise straightforward application of the method of the present article. In any case, we acknowledge that the patchwork helped us recognize that, for the purposes of the present work, to consider $d$ lines (\autoref{Lines}) makes things simpler than considering $d-1$ lines (\cite[Proposition~2.7]{BD}).

\section{Preparation}
 
We work over an algebraically closed field $\K$ of characteristic zero and fix two symmetric $\K$-algebras $S^\bullet=\Sy^\bullet S^1$, $S_\bullet=\Sy^\bullet S_1$; we shall keep this notation throughout the paper. We also assume that an \emph{apolarity pairing} between $S^\bullet,S_\bullet$ is given. It is naturally induced by a perfect pairing $S^1\times S_1\to\K$ (for more details see \cite[Introduction]{D1}). This amounts to say that $S^\bullet$, $S_\bullet$ are rings of polynomials in a finite and the same number of indeterminates, acting on each other by constant coefficients partial differentiation. For each $x\in S^\bullet$ and $f\in S_\bullet$ we shall denote by $\partial_xf$ the apolarity action of $x$ on $f$. For each form (homogeneous polynomial) $f\in S_{d+\delta}$, we shall denote by $f_{\delta,d}$ the \emph{partial polarization} map $S^\delta\to S_d$ defined by $f_{\delta,d}(x):=\partial_xf$. The \emph{apolar ideal} of $f\in S_d$ is the set of all $x\in S^\bullet$ such that $\partial_xf=0$. We also define the evaluation of a homogeneous form $x\in S^d$ on a linear form $v\in S_1$, by setting
\[
x(v):=\frac{\partial_xv^d}{d!}\,.
\]
The (Waring) \emph{rank} of $f\in S_d$, $d>0$, denoted by $\rk f$, is the least of the numbers $r$ such that $f$ can be written as a sum of $r$ $d$th powers of forms in $S_1$ (\footnote{Since we are assuming that $\K$ is algebraically closed, when $d>0$ a form $f\in S_d$ is a sum of $r$ $d$th powers of linear forms if and only if it is a linear combination of $r$ $d$th powers of linear forms. Using linear combinations allows one to define Waring rank in degree $0$ as well, and of course it would be $1$ for every nonzero constant. We prefer not to decide here whether the rank of a nonzero constant should be $1$ or be left undefined.}); $\rmax(n,d)$ is the maximum of the ranks of all such $f$ when $\dim S_1=n$. The span of $v_1,\ldots,v_r$ in some vector space $V$ will be denoted by $\Span{v_1,\ldots,v_r}$, and the projective space made of all one-dimensional subspaces $\Span{v}\subseteq V$, $v\ne 0$, will be denoted by $\Ps V$. A \emph{morphism of projective spaces $\Ps\varphi:\Ps V\smallsetminus\Ps\Ker\varphi\to\Ps W$} is a map determined by a linear map $\varphi:V\to W$ by setting $\Ps\varphi\left(\Span{v}\right):=\Span{\varphi(v)}$. The sign $\perp$ will refer to orthogonality with respect to the apolarity pairing $S^d\times S_d\to\K$, when some degree $d$ is fixed (sometimes implicitly).

In \cite[Sec.~1.3]{IK}, building on classical results due to Sylvester, the authors deal with binary forms (i.e. $\dim S_1=2$, in our notation). They show that power sum decompositions are closely related with the initial degree of the (homogeneous) apolar ideal, that is, the least degree of a nonzero homogeneous element of that ideal. That is the notion of \emph{length} of a binary form (see \cite[Def.~1.32 and Lemma~1.33]{IK}), which can be generalized in various ways for forms in more indeterminates: see \cite[Def.~5.66]{IK}. Nowadays, terms related to length are replaced by similar terms related with rank, probably because of the renewed interest in the interplay with the rank of tensors. In the present paper we need that notion only when the form is essentially binary, and what we really use is only its algebraic property of being the initial degree of the apolar ideal in a ring of binary forms. Note that a form $f\in S_d$ belongs to some subring $T_\bullet=\Sy^\bullet T_1$ with $\dim T_1=2$, if and only if $\Ker f_{1,d-1}$ has codimension at most $2$ in $S^1$ (it suffices to take a two-dimensional $T_1\supseteq{\Ker f_{1,d-1}}^\perp$). Moreover, $f$ belongs to more than one of such subrings if and only if $\Ker f_{1,d-1}$ has codimension at most $1$, in which case the initial degree of the apolar ideal of $f$ in each of the subrings $T_\bullet$, whatever dual ring $T^\bullet$ one chooses,  is always the same (and equal to the codimension). This allows us to state the following definition.

\begin{defn}
Let $f\in S_d$. If $f$ belongs to some ring $T_\bullet=\Sy^\bullet T_1$, contained or containing $S_\bullet$ (with the ring structures preserved), with $\dim T_1=2$, then we define the \emph{binary length of $f$} as the initial degree of its apolar ideal, and we denote it by~$\blen f$.
\end{defn}

The following definition is also useful.

\begin{defn}
Given $x\in S^\bullet$ and $f\in S_\bullet$, an \emph{$x$-antiderivative of $f$} is a polynomial $F\in S_\bullet$ such that $\partial_xF=f$, and when $x,f$ are nonzero and homogeneous we sometimes also say that $\Span{F}$ is a \emph{$\Span{x}$-antiderivative of $\Span{f}$}. Moreover, still in the homogeneous case $x\in S^\delta$, $f\in S_d$, if a decomposition
\begin{equation}\label{D}
f=\lambda_1{v_1}^d+\cdots+\lambda_r{v_r}^d\;,\quad\lambda_1,\ldots,\lambda_r\in\K,\;v_1,\ldots,v_r\in S_1\;,
\end{equation}
is given and $x$ vanishes on no one of $v_1,\ldots ,v_r$, then we define \emph{the} $x$-antiderivative of $f$ (\footnote{Or also the $\Span{x}$-antiderivative of $\Span{f}$, if $f\ne 0$.}) \emph{relative to \eqref{D}} as the form
\[
F:=\frac{d!\lambda_1}{(d+\delta)!x\left(v_1\right)}{v_1}^{d+\delta}+\cdots+\frac{d!\lambda_r}{(d+\delta)!x\left(v_r\right)}{v_r}^{d+\delta}\;;
\]
when the powers ${v_1}^d,\ldots ,{v_r}^d$ are linearly independent we also say that the above $x$-antiderivative is \emph{relative to $v_1,\ldots ,v_r$}.
\end{defn}

Note that, in the above notation, the $x$-antiderivative relative to $v_1,\ldots ,v_r$ is the unique $x$-antiderivative of $f$ that lies in $\Span{{v_1}^{d+\delta},\ldots ,{v_r}^{d+\delta}}$.
 
Now we explicitly point out two basic facts that are probably well-known, but for which we are not aware of a direct reference (\footnote{Basic facts like these are heavily scattered in the litterature, and we may easily overlook some reference. For instance, \cite[Remark~3.6]{D2} could have been avoided by using the original reference \cite[Theorem~2]{CS} instead of \cite[9.2.2.1]{L}. We take this occasion for apologizing for that.}).

\begin{rem}\label{Div}
Let $h\in S^d$, $h'\in S^{d'}$ and $e\ge d$. Then $h'$ divides $h$ if and only if $S^e\cap\Ker\partial_{h'}\subseteq S^e\cap\Ker\partial_h$. One implication immediately follows from $\partial_{ph'}=\partial_p\circ\partial_{h'}$, and by the same reason we have $S^e\cap\Ker_{\partial_{h'}}\subseteq\left(h'S^{e-d'}\right)^\perp$. When $h'\ne 0$, since $\partial_{h'}$ maps $S^e$ \emph{onto} $S^{e-d'}$ we have $\dim S^e\cap\Ker\partial_{h'}=\dim S^e-\dim S^{e-d'}$. Since the apolarity pairing is nondegenerate in fixed degree, we also have $\dim\left(h'S^{e-d'}\right)^\perp=\dim S^e-\dim S^{e-d'}$ when $h'\ne 0$. Hence $S^e\cap\Ker\partial_{h'}=\left(h'S^{e-d'}\right)^\perp$ (even when $h'=0$) and, similarly, $S^e\cap\Ker\partial_h=\left(hS^{e-d}\right)^\perp$. Now, to show the converse implication, let us suppose that $S^e\cap\Ker\partial_{h'}\subseteq S^e\cap\Ker\partial_h$, that is, $\left(h'S^{e-d'}\right)^\perp\subseteq\left(hS^{e-d}\right)^\perp$. Again because apolarity is nondegenerate in fixed degree, we deduce that $hS^{e-d}\subseteq h'S^{e-d'}$. Choosing $x\in S^1$ that does not divide $h'$ (we can assume $h'\ne 0$ and $\dim S_1\ge 2$, since the proof is trivial in the opposite case), we have that $h'$ divides $x^{e-d}h$, hence $h'$ divides $h$.
\end{rem}

\begin{rem}\label{Power}
Let $f\in S^d$ with $d\ge 3$ and $\dim S_1=3$. Then $f$ is a $d$th power if and only if for each $x\in S^1$, $\partial_xf$ is a $(d-1)$th power. One implication is immediate. Conversely, suppose that for each $x\in S^1$, $\partial_xf$ is a $(d-1)$th power of some linear form. Let $x,y\in S^1$ be linearly independent. According to \cite[Lemma~4.1]{D2} we have that $\partial_lf=0$ for some nonzero $l\in\Span{x,y}$ (\footnote{In the statement of the cited result, by mistake, one finds $f\in S_d$ instead $f\in S_{d+1}$ and  the condition $d\ge 2$ is omitted; cf.~\cite[Rem.~2.2]{D2}.}). Repeating the argument with another two-dimensional subspace of $S^1$ that does not contain $l$, we get $m\in S^1$ such that $\partial_m=0$ and with $l,m$ linearly independent. Hence $\Ker f_{1,d-1}\ge 2$, and since $\dim S_1=3$ we have that $f$ is a $d$th power of a linear form.

Moreover, let $f\in S_d$, $d\ge 3$, be not a $d$th power. The set of all $\Span{v^{d-1}}$ with $\Span{v}\in\Ps S_1$ is an algebraic (Veronese) variety in $\Ps S_{d-1}$, and $\Ps f_{1,d-1}$ is an algebraic morphism from a nonempty Zariski open subset of $\Ps S^1$ to $\Ps S_{d-1}$. Then $\partial_xf$ is not a $(d-1)$th power for all $\Span{x}$ in a nonempty open Zariski subset of $\Ps S^1$.
\end{rem}

We end this section with two technical lemmas.
 
\begin{lemma}\label{H}
Let $\Span{f}\in\Ps S_d$ and $\Span{x}\in\Ps S^1$, with $\dim S_1=2$. Let $I$ be the apolar ideal of $f$ and set $\ell:=\blen f$, $\ell':=d+2-\ell$,
\begin{equation}\label{Pmor}
W:=S_{d+1}\cap{\partial_{x}}^{-1}\left(\Span{f}\right)\;,\quad H:=S^{\ell'}\cap I\;,\quad K:=S^{\ell'}\cap xI\;.
\end{equation}
Finally, let $X$ be the locus of all $\Span{h}\in\Ps H$ such that $h$ is not squarefree and set $\Span{v_\infty}:=\Span{x}^\perp$, so that \[\Span{{v_\infty}^{d+1}}=S_{d+1}\cap\Ker\partial_x\subset W\;.\]
Then
\begin{itemize}
\item there exists an epimorphism of projective spaces \[\omega:\Ps H\smallsetminus\Ps K\to\Ps W\;,\quad\omega\left(\Span{h}\right)=:\Span{w_h}\;,\] such that $\partial_hw_h=0$ for all $\Span{h}$;
\item for all $\Span{w}\in\Ps W\smallsetminus\Span{{v_\infty}^{d+1}}$ but at most one, we have \[\blen w=\min\left\{\ell+1,\ell'\right\}\;;\]
\item $X\subsetneq\Ps H$;
\item for each projective line $\Ps L\subseteq\Ps H$ that does not meet $\Ps K$, the restriction $\Ps L\to\Ps W$ of $\omega$ is an isomorphism of projective spaces, and if the line $\Ps L$ is not contained in $X$ then there exists a cofinite subset $U\subset\Ps L$ such that for each $\Span{h}\in U$ we have
\begin{itemize}
\item $h$ has distinct roots $\Span{v_1},\ldots,\Span{v_{\ell'}}\in\Ps S_1$;
\item $f\in\Span{{v_1}^d,\ldots ,{v_{\ell'}}^d}$;
\item $x$ vanishes on no one of $v_1,\ldots, v_{\ell'}$ and $\Span{w_h}$ is the $\Span{x}$--antiderivative of $\Span{f}$ relative to $v_1,\ldots ,v_{\ell'}$.
\end{itemize}
\end{itemize}
\end{lemma}
\begin{proof}
For each $h\in H$ and $w\in W$ we have $\partial_{xh}w=0$, because $\partial_xw\in\Span{f}$ and $h\in I$; hence $\partial_hw\in S_{\ell-1}\cap\Ker\partial_x=\Span{{v_\infty}^{\ell-1}}$. Thus we have a bilinear map
\[
\beta:H\times W\to\Span{{v_\infty}^{\ell-1}}\;,\qquad\beta(h,w):=\partial_hw\;.
\]
If $h\in K$, then $h=xh'$ for some $h'\in I$; hence for all $w\in W$ we have $\partial_hw=\partial_{h'}\partial_xw=0$, because $\partial_xw\in\Span{f}$. This shows that $K$ is contained in the left kernel of $\beta$. Conversely, if $h$ is in the left kernel, then $\partial_h$ vanishes on $W$, and in particular on $\Span{{v_\infty}^{d+1}}\subset W$. Hence $h=xh'$ for some $h'\in S^d$, by \autoref{Div} (\footnote{Alternatively, one may observe that $0=\partial_h\left({v_\infty}^{d+1}\right)=\frac{(d+1)!}{\ell'!}h\left(v_\infty\right){v_\infty}^{\ell-1}$. Hence $h$ vanishes on the root $v_\infty$ of $x$, that is, $h$ is divisible by $x$.}). Choosing an $x$-antiderivative $w$ of $f$, we have $0=\partial_hw=\partial_{h'}f$, and thus $h'\in I$. We conclude that $K$ is the left kernel of $\beta$.

Let \[\overline{\beta}:H\to\Hom\left(W,\Span{{v_\infty}^{\ell-1}}\right)\;,\quad\overline{\beta}(h)(w):=\beta(h,w)=\partial_hw\;,\] be the homomorphism induced by $\beta$, and let $\iota:W\to \Hom\left(W,\Span{{v_\infty}^{\ell-1}}\right)$ be an isomorphism such that $\iota(w)(w)=0$ for all $w\in W$ (in other words, $\iota$ is the homomorphism induced by a nondegenerate bilinear alternating map on $W$ with values in $\Span{{v_\infty}^{\ell-1}}$, which certainly exists because $\dim W=2$). Then $\varphi:=\iota^{-1}\circ\overline{\beta}:H\to W$ is a linear map with kernel $K$ such that $\partial_h\left(\varphi(h)\right)=0$ for all $h\in H$. This shows that $\omega:=\Ps\varphi$ is a morphism of projective spaces such that $\partial_hw_h=0$ (under the notation $\Span{w_h}:=\omega\left(\Span{h}\right)=\Span{\varphi(h)}$). We have to check that $\omega$ is surjective.

According to \cite[Theorem~1.44(iv)]{IK}, $I$ is generated by two homogeneous forms $l\in S^\ell, h^0\in S^{\ell'}$ (hence $h^0\in H$). Recall also that $\ell\le\ell'$ because $\ell=\blen f$. Therefore
\begin{equation}\label{Descr}
H=lS^{\ell'-\ell}+\Span{h^0}\;,\qquad K=lxS^{\ell'-\ell-1}\;.
\end{equation}
Since $S^{d+1}\subset I$, we have that $l,h^0$ are coprime, and therefore $h^0\not\in lS^{\ell'-\ell}$. Since $\dim\varphi\left(lS^{\ell'-\ell}\right)=1$ we have that $\varphi$ is surjective, and hence $\omega$ is surjective as it was to be shown.

Let $\varphi\left(lS^{\ell'-\ell}\right)=:\Span{w_l}\in\Ps W$ (possibly $\Span{w_l}=\Span{{v_\infty}^{d+1}}$). Since $\partial_{lp}w_l=0$ for all $p\in S^{\ell'-\ell}$, we have $\partial_lw_l=0$. Note also that $w_l=w_{pl}$ for all $p\in S^{\ell'-\ell}\smallsetminus xS^{\ell'-\ell-1}$. Moreover,
\begin{equation}\label{Res}
\Span{h},\Span{h'}\in\Ps H\smallsetminus\Ps  K\;,\Span{w_{h'}}\ne\Span{w_h}\;\Longrightarrow\;\partial_hw_{h'}\ne 0\;,
\end{equation}
because $\partial_hw_{h'}=0$ would imply that $\partial_h$ vanishes on $\Span{w_h,w_{h'}}=W$ ($\dim W=2$), and this is excluded since $h\not\in K$. Since $\omega$ is suriective, we conclude that $\partial_lw\ne 0$ for each $\Span{w}\in\Ps W\smallsetminus\left\{\Span{w_l}\right\}$. On the other hand, if $\Span{w}\in\Ps W\smallsetminus\left\{\Span{{v_\infty}^{d+1}}\right\}$, then $\Span{\partial_xw}=\Span{f}$, and hence the apolar ideal of $w$ is contained in $I$ and contains~$xI$. Thus \[\ell\le\blen w\le\ell+1\;,\qquad \forall\Span{w}\in\Ps W\smallsetminus\left\{\Span{{v_\infty}^{d+1}}\right\}.\]

Now, if $\ell'\ge\ell+1$, then for each $\Span{w}\in\Ps W\smallsetminus\left\{\Span{{v_\infty}^{d+1}},\Span{w_l}\right\}$ we have $\blen w=\ell+1=\min\left\{\ell+1,\ell'\right\}$. To deal with the case $\ell'=\ell$, notice that for each $w\in W$ we have $\partial_{h}w=0$ for some $\Span{h}\in H$, because $\varphi$ is surjective; hence $\blen w\le\ell'$. Thus, if $\ell=\ell'$ then for each $\Span{w}\in\Ps W\smallsetminus\Span{{v_\infty}^{d+1}}$ we have $\blen w=\ell'=\min\left\{\ell+1,\ell'\right\}$ (\footnote{The equality $\blen w=\min\left\{\ell+1,\ell'\right\}$ we have just proved for all $\Span{w}\in\Ps W\smallsetminus\left\{\Span{{v_\infty}^{d+1}},\Span{w_l}\right\}$ says, in other terms, that $\blen w=\ell+1$ unless $d$ is even, $d=2s$, and $\ell$ is the maximum allowed for that degree, that is, $s+1$.}).

Since $l,h^0$ are coprime, taking into account \eqref{Descr} and Bertini's theorem (see also \cite[Lemma~1.1, Remark~1.1.1]{K}), we have that $X$ is a proper subset of $\Ps H$.

Finally, let $\Ps L\subseteq\Ps H\smallsetminus\Ps K$ be a projective line. The restriction $\Ps L\to\Ps W$ of $\omega$ is an isomorphism simply because $\Ps W$ is a projective line as well, and $\Ps L\cap\Ps K=\emptyset$. Since the proper subset $X\subsetneq\Ps H$ is algebraic, with equation given by the discriminant of degree $\ell'$ forms (inside $\Ps H$), we have that if $\Ps L\not\subseteq X$, then $U:=\Ps L\smallsetminus\left(X\cup\omega^{-1}\left(\Span{{v_\infty}^{d+1}}\right)\right)$ is a cofinite subset of $\Ps L$. Since each $\Span{h}\in U$ is outside $X$, $h$ is squarefree, that is, it has distinct roots $\Span{v_1},\ldots,\Span{v_{\ell'}}\in\Ps S_1$. For such $h,v_1,\ldots ,v_{\ell'}$, according to \cite[Lemma~1.31]{IK}, we have $f\in\Span{{v_1}^d,\ldots ,{v_{\ell'}}^d}$ as required. By the same reason, we have $w_h\in\Span{{v_1}^{d+1},\ldots ,{v_{\ell'}}^{d+1}}$, and since $\Span{w_h}\ne\Span{{v_\infty}^{d+1}}$, $\Span{w_h}$ is a $\Span{x}$--antiderivative of $\Span{f}$. Moreover, $x$ vanishes on no one of $v_1,\ldots, v_{\ell'}$ by \eqref{Res}, and ${v_1}^d,\ldots ,{v_{\ell'}}^d$ are linearly independent because $\ell'\le d+1$. The above said suffices to prove that $\Span{w_h}$ is the $\Span{x}$--antiderivative of $\Span{f}$ relative to $v_1,\ldots ,v_{\ell'}$.
\end{proof}

\begin{lemma}\label{T}
Let $\Span{g'}\in\Ps S_d$ with $\dim S_1=3$, $0<d=2s+\varepsilon$, $\varepsilon\in \{0,1\}$ and $s$ integer. Let $\Span{l^0},\ldots,\Span{l^t}\in\Ps S^1$ be distinct and such that $\partial_{l^0}g'=0$, and for each $i\in\{1,\ldots, t\}$ let $g_i$ be an $l^i$-antiderivative of $g'$. If 
\[
\blen g'=\blen\partial_{l^0}g_1=\cdots=\blen\partial_{l^0}g_t=s+1
\]
{\rm(\footnote{For each $i$, $\partial_{l^0}g_i$ is annihilated by $l^i$.})}, then there exists a power sum decomposition
\begin{equation}\label{Decv}
g'={v_1}^d+\cdots+{v_r}^d
\end{equation}
such that: $r\le s+1+\varepsilon$ and, for each $i\in\{1,\ldots, t\}$,
\begin{itemize}
\item $l^i$ vanishes on no one of $v_1,\ldots, v_r$,
\item denoting by $F_i$ the $l^i$-antiderivative relative to \eqref{Decv}, $\blen\left(g_i-F_i\right)=s+1+\varepsilon$.
\end{itemize}
\end{lemma}
\begin{proof}
For each $i\in\{0,\ldots, t\}$, let $R_i^\bullet:=S^\bullet/\left(l^i\right)$, $R_{i,\bullet}:=\Ker\partial_{l^i}\subset S_\bullet$, with the apolarity pairing induced by the one between $S^\bullet$ and $S_\bullet$. Let $I\subset R_0^\bullet$ be the apolar ideal of $g'\in R_{0,d}$, set
\[
H:=R_0^{s+1+\varepsilon}\cap I\;,
\]
for each $i\in\{1,\ldots, t\}$ set
\[
W_{0,i}:=R_{0,d+1}\cap{\partial_{l^i}}^{-1}\left(\Span{g'}\right)
\]
and when $\varepsilon=1$, also $\Span{k_i}:=R_0^{s+2}\cap l^iI$. For each $i\in\{1,\ldots, t\}$, let us exploit \autoref{H} with $R_0^\bullet$, $R_{0,\bullet}$, $g'$, $l^i+\left(l^0\right)$ in place of $S^\bullet$, $S_\bullet$, $f$, $x$. We get projective epimorphisms \[\omega_i:\Ps H\to\Ps W_{0,i}\;.\]
Moreover, we can fix a projective line $\Ps L\subseteq\Ps H$ ($\Ps L=\Ps H$ when $\varepsilon=0$) not contained in the singular locus $X$ (which does not depend on $i$) and passing through no one of $\Span{k_1},\ldots,\Span{k_t}$ (when $\varepsilon=1$). Hence, the restriction $\varrho_i:\Ps L\to\Ps W_{0,i}$ of $\omega_i$ is a projective isomorphism for each $i$, and we also have cofinite subsets $U_{0,i}\subset\Ps L$ that fulfill the properties listed by the end of the statement of \autoref{H}.

Now, for each $i$ we have $\partial_{l^0}g_i\ne 0$ because $\blen\partial_{l^0}g_i=s+1>0$. Hence the vector space
\[
W_{i,0}:=R_{i,d+1}\cap{\partial_{l^0}}^{-1}\left(\Span{\partial_{l^0}g_i}\right)
\]
is two-dimensional. Since $W_{0,i}=R_{0,d+1}\cap{\partial_{l^i}}^{-1}\left(\Span{g'}\right)$, for all $w\in W_{0,i}$ we have 
\[
\partial_{l^i}w=\lambda_i(w)g'
\]
for some scalar $\lambda_i(w)$, and therefore $\lambda_i(w)g_i-w\in W_{i,0}$. This defines a map $W_{0,i}\to W_{i,0}$ and to check that it is a vector space isomorphism is easy (take into account that $\partial_{l^0}\left(\lambda_i(w)g_i-w\right)=\lambda_i(w)\partial_{l^0}g_i$). Therefore we have isomorphisms of projective spaces
\[
\tau_i:\Ps W_{0,i}\to\Ps W_{i,0}\;,\qquad \Span{w}\mapsto\Span{\lambda_i(w)g_i-w}\;.
\]
According to \autoref{H}, we have cofinite subsets $U'_{i,0}\subset\Ps W_{i,0}$ such that
\begin{equation}\label{Blup}
\blen w=s+1+\epsilon\;,\quad \forall\Span{w}\in U'_{i,0}\;
\end{equation}
(more precisely, $\sharp\left(\Ps W_{i,0}\smallsetminus U'_{i,0}\right)\le 2$).

Let $U_{i,0}:={\varrho_i}^{-1}\left({\tau_i}^{-1}\left(U'_{i,0}\right)\right)$ for each $i$, which is obviuosly a cofinite subset of~$\Ps L$. Now, let us pick $\Span{h}$ in the nonempty intersection
\[
U_{0,1}\cap\cdots\cap U_{0,t}\cap U_{1,0}\cap\cdots\cap U_{t,0}\;,
\]
and let $\Span{v_1},\ldots,\Span{v_{s+1+\epsilon}}$ be its roots, which are distinct because $\Span{h}\in U_{0,i}$ (whatever $i$ one chooses). For each $i$, $l^i$ vanishes on no one of $v_1,\ldots, v_{s+1+\epsilon}$, because $\Span{h}\in U_{0,i}$. Since $g'\in\Span{{v_1}^d,\ldots,{v_{s+1+\varepsilon}}^d}$, $d>0$, for an appropriate choice of the representatives $v_1,\ldots, v_{s+1+\epsilon}$ one gets \eqref{Decv}. Since $F_i$ is the $l^i$-antiderivative of $g'$ relative to \eqref{Decv}, that is, relative to $v_1,\ldots, v_{s+1+\epsilon}$, we have $\Span{F_i}=\omega_i\left(\Span{h}\right)$. Since $F_i$ is an $l^i$-antiderivative of $g'$, we have $\lambda_i\left(F_i\right)=1$, and hence \[\tau_i\left(\omega_i\left(\Span{h}\right)\right)=\Span{g_i-F_i}\;.\] Since $\Span{h}\in U_{i,0}$ for each $i$, we have $\Span{g_i-F_i}\in U'_{i,0}$, and therefore $\blen\left(g_i-F_i\right)=s+\varepsilon+1$ by \eqref{Blup}.
\end{proof}

\section{The upper bound}

\begin{prop}\label{Lines}
Let $f\in S_d$, with $\dim S_1=3$ and $d\ge 2$, let $f_1,\ldots, f_a\in S_\bullet$ be homogeneous forms with degrees at least $d+1$, and let $\Span{x^1},\ldots,\Span{x^b}\in\Ps S^1$. If $f, f_1,\ldots, f_a$ are not powers of linear forms, then there exist distinct
\[
\Span{l^1},\ldots,\Span{l^d}\in\Ps S^1\smallsetminus\left\{\Span{x^1},\ldots,\Span{x^b}\right\}
\]
such that
\[
\partial_{l^1\cdots l^d}f=0\;,\quad\partial_{l^1\cdots\widehat{l^i}\cdots l^d}f\ne 0\;\forall i\in\{1,\ldots ,d\},\quad\partial_{l^1\cdots l^d}f_j\ne 0\;\forall j\in\{1,\ldots ,a\}\,,
\]
where the hat denotes omission.
\end{prop}
\begin{proof}
According to \autoref{Power}, we can fix
\[
\Span{l^d}\in\Ps S^1\smallsetminus\left\{\Span{x^1},\ldots,\Span{x^b}\right\}
\]
such that $f'_1:=\partial_{l^d}f_1,\ldots ,f'_a:=\partial_{l^d}f_a$ are not powers of linear forms. Since $f\ne 0$ because it is not a $d$th power, we can also assume that, in addition, $f':=\partial_{l^d}f\ne 0$.

Let us first suppose that $d=2$. For each $i\in\{1,\ldots ,a\}$, since $f'_i$ is not a power of a linear form, we may have $\partial_lf'_i=0$ for at most one $\Span{l}\in\Ps S^1$. Therefore there exists a finite subset $\Sigma\subset\Ps S^1$ such that for all $\Span{l}\in\Ps S^1\smallsetminus\Sigma$ and $i\in\{1,\ldots ,a\}$ we have $\partial_lf'_i\ne 0$. We have $\partial_lf=0$ for at most one $\Span{l}\in\Ps S^1$ as well. Since $\Ker f'_{1,0}=\Span{f'}^\perp$ is an infinite set, we can pick out $\Span{l^1}\in\Ps\Span{f'}^\perp$ such that
\begin{itemize}
\item $\partial_{l^1}f\ne 0$;
\item $\Span{l^1}\not\in\left\{\Span{l^2},\Span{x^1},\ldots,\Span{x^b}\right\}\cup\Sigma$.
\end{itemize}
It is immediate to check that $\Span{l^1},\Span{l^2}$ fulfill all the requirements in the statement.

Now, let us assume $d\ge 3$. In this case we can assume that $f'$ is not a $(d-1)$th power and, by induction on $d$, that the statement holds with $f'$ in place of $f$, with $f,f'_1,\ldots,f'_a$ in place of $f_1,\ldots,f_a$ and with $l^d,x^1,\ldots, x^b$ in place of $x^1,\ldots, x^b$. This gives linear forms $l^1,\ldots ,l^{d-1}$ that, together with $l^d$, fullfill all the requirements.
\end{proof}

\begin{prop}\label{Main}
Let $f\in S_d$ with $\dim S_1=3$, $e\in\{0,\ldots,d\}$, $e=2s+\varepsilon$, with $\varepsilon\in \{0,1\}$ and $s$ integer, and let \[\Span{l^1},\ldots,\Span{l^d}\in\Ps S^1\] be distinct and such that \[\partial_{l^1\cdots l^d}f=0\;;\qquad\partial_{l^1\cdots \widehat{l^i}\cdots l^d}f\ne 0\;\forall i\;.\]
Then there exists a power sum decomposition
\begin{equation}\label{Dec}
\partial_{l^{e+1}\cdots l^d}f={v_1}^{e}+\cdots+{v_r}^{e}
\end{equation}
such that:
\begin{itemize}
\item $r\le s^2+3s+\varepsilon(s+2)$;
\item for each $i\in\{e+1,\ldots ,d\}$, $l^i$ vanishes on no one of $v_1,\ldots ,v_r$ and denoting by $F_i$ the $l^i$-antiderivative relative to \eqref{Dec}, we have \[\blen\left(\partial_{l^{e+1}\cdots\widehat{l^i}\cdots l^d}f-F_i\right)=s+1+\varepsilon\;.\]
\end{itemize}
\end{prop}
\begin{proof}
When $e=0$ it suffices to define \eqref{Dec} as the decomposition of $0$ with no summands. By induction, we can assume that $e\ge 1$ and that the proposition holds with $e-1$ in place of $e$. Therefore we get a decomposition
\begin{equation}\label{Decp}
\partial_{l^e\cdots l^d}f={v'_1}^{e-1}+\cdots+{v'_{r'}}^{e-1}
\end{equation}
such that
\begin{equation}\label{rp}
r'\le s^2+2s-1+\varepsilon(s+1)
\end{equation}
and each of $l^e, \ldots, l^d$ vanishes on no one of $v'_1, \ldots, v'_{r'}$. We can also consider for each $i\in\{e,\ldots,d\}$ the $l^i$-antiderivative relative to \eqref{Decp}, which we denote by $G'_i$, and set
\[
g'_i:=\partial_{l^e\cdots\widehat{l^i}\cdots l^d}f-G'_i\;,
\]
so that \[\blen g'_i=s+1\;.\]
For each $i\in\{e+1,\ldots,d\}$, let $G_i$ be the $l^el^i$--antiderivative relative to \eqref{Decp} and set
\[
g_i:=\partial_{l^{e+1}\cdots\widehat{l^i}\cdots l^d}f-G_i\;,
\]
so that \[\partial_{l^e}g_i=g'_i\;,\quad\partial_{l^i}g_i=g'_e\;.\]

By the above construction, we can exploit \autoref{T} with $e$, $g'_e$, $l^e,\ldots,l^d$, $g_{e+1},\dots ,g_d$ in place of $d$, $g'$, $l^0,\ldots,l^t$, $g_1,\ldots,g_t$. We get a decomposition
\begin{equation}\label{Deca}
g'_e={v_1}^e+\cdots+{v_{r''}}^e
\end{equation}
such that
\begin{equation}\label{rs}
r''\le s+1+\varepsilon\;,
\end{equation}
each of $l^{e+1},\ldots,l^d$ vanishes on no one of $v_1,\ldots ,v_{r''}$ and denoting by $H_i$ the $l^i$-antiderivative relative to \eqref{Deca}, we have
\begin{equation}\label{Cbl}
\blen\left(g_i-H_i\right)=s+1+\varepsilon\;.
\end{equation}

Since we defined $G'_e$ as the $l^e$-antiderivative relative to \eqref{Decp}, by taking suitable multiples of $v'_1,\ldots, v'_{r'}$ and calling them $v_{r''+1},\ldots, v_{r''+r'}$, respectively, we have
\[
G'_e={v_{r''+1}}^e+\cdots+{v_r}^e\;,
\]
with $r:=r''+r'$. By definition of $g'_e$ and by \eqref{Deca} we conclude that
\[
\partial_{l^{e+1}\cdots l^d}f=g'_e+G'_e={v_1}^e+\cdots+{v_r}^e\;.
\]
To show that the above is the required decomposition \eqref{Dec}, first note that \eqref{rp} and \eqref{rs} give
\[
r\le s^2+3s+\varepsilon(s+2)\;,
\]
as it was to be shown. Moreover, since each of $l^e, \ldots, l^d$ vanishes on no one of $v'_1, \ldots, v'_{r'}$, which are proportional to $v_{r''+1},\ldots, v_r$, and each of $l^{e+1},\ldots,l^d$ vanishes on no one of $v_1,\ldots ,v_{r''}$, we have that for each $i\in\{e+1,\ldots ,d\}$, $l^i$ vanishes on no one of $v_1,\ldots ,v_r$. Finally, for the $l^i$-antiderivatives $F_i$s we have
\[
F_i=H_i+G_i=H_i+\partial_{l^{e+1}\cdots\widehat{l^i}\cdots l^d}f-g_i\;,
\]
hence
\[
\partial_{l^{e+1}\cdots\widehat{l^i}\cdots l^d}f-F_i=g_i-H_i
\]
and the last requirement to be fulfilled follows from \eqref{Cbl}.
\end{proof}

\begin{prop}\label{Bound}
When $\dim S_1=3$, $d>0$, for all $f\in S_d$ we have
\[
\rk f\le\left\lfloor\frac{d^2+6d+1}4\right\rfloor\;.
\]
\end{prop}
\begin{proof}
If $f$ is a $d$th power then $\rk f\le 1$ and the result trivially follows. Hence we can assume that $f$ is not a $d$th power. Exploiting \autoref{Lines} with $a=b=0$, we get $\Span{l^1},\ldots,\Span{l^d}\in\Ps S^1$ such that
\[
\partial_{l^1\cdots l^d}f=0\;,\qquad\partial_{l^1\cdots\widehat{l^i}\cdots l^d}f\ne 0\;\forall i\in\{1,\ldots ,d\}\;.
\]
Now the result immediately follows from \autoref{Main} with $e:=d$.
\end{proof}

\begin{prop}
We have $\rmax(3,d)=d^2/4+O(d)$.
\end{prop}
\begin{proof}
An immediate consequence of \autoref{Bound} together with \cite[Proposition~4.1]{CCG} (see also \cite[Theorem~7]{BBT}, \cite[Theorem~1]{BuT}).
\end{proof}

\end{document}